\documentclass{amsart}

\usepackage{amssymb}
\usepackage{graphicx}

\newtheorem{theorem}{Theorem}
\newtheorem{lemma}{Lemma}[section]
\newtheorem{prop}{Proposition}[section]
\newtheorem{cor}{Corollary}

\theoremstyle{definition}

\theoremstyle{remark}
\newtheorem{remark}{Remark}[section]

\numberwithin{equation}{section}

\newcommand{\abs}[1]{\left\lvert#1\right\rvert}
\newcommand{\norm}[1]{\left\lVert#1\right\rVert}

\newcommand{\CA}{\mathcal{A}}
\newcommand{\CB}{\mathcal{B}}
\newcommand{\CC}{\mathcal{C}}

\newcommand{\CE}{\mathcal{E}}
\newcommand{\CH}{\mathcal{H}}
\newcommand{\CI}{\mathcal{I}}
\newcommand{\CJ}{\mathcal{J}}

\newcommand{\CM}{\mathcal{M}}

\begin{document}

\title[Distribution of the number of points on curves]
{On the distribution of the number of points on a family of curves over finite fields}

\author{Kit-Ho Mak}
\address{Department of Mathematics \\
University of Illinois at Urbana-Champaign \\
273 Altgeld Hall, MC-382 \\
1409 W. Green Street \\
Urbana, Illinois 61801, USA}
\email{mak4@illinois.edu}

\author{Alexandru Zaharescu}
\address{Department of Mathematics \\
University of Illinois at Urbana-Champaign \\
273 Altgeld Hall, MC-382 \\
1409 W. Green Street \\
Urbana, Illinois 61801, USA}
\email{zaharesc@math.uiuc.edu}


\thanks{The second author is supported by NSF grant number DMS - 0901621.}

\begin{abstract}
Let $p$ be a large prime, $\ell\geq 2$ be a positive integer, $m\geq 2$ be an integer relatively prime to $\ell$ and $P(x)\in\mathbb{F}_p[x]$ be a polynomial which is not a complete $\ell'$-th power for any $\ell'$ for which $GCD(\ell',\ell)=1$. Let $\mathcal{C}$ be the curve defined by the equation $y^{\ell}=P(x)$, and take the points on $\mathcal{C}$ to lie in the rectangle $[0,p-1]^2$. In this paper, we study the distribution of the number of points on $\mathcal{C}$ inside a small rectangle among residue classes modulo $m$ when we move the rectangle around in $[0,p-1]^2$.
\end{abstract}

\maketitle

\section{Introduction}

Since Weil's proof of the Riemann hypothesis for algebraic curves over finite fields \cite{Wei48b}, there have been numerous studies on the number of rational points of an algebraic curve over a finite field in a specified set of number theoretic interest. Examples include studies of bounds on the number of rational points in a smaller region inside $[0,p-1]^2$ (see for example Myerson \cite{Mye81}, Fujiwara \cite{Fuj88}, and \cite{MaZa10}), bounds on the number of points in sets with prescribed congruence conditions on the coordinates (known as Lehmer problems, see for example Zhang \cite{Zha93, Zha94}, Cobeli and one of the authors \cite{CoZa01} and Bourgain, Cochrane, Paulhus and Pinner \cite{BCPP11}), bounds on the number of visible points (see Shparlinski \cite{Shp06}, Shparlinski and Voloch \cite{ShVo07}, Shparlinski and Winterhof \cite{ShWi08}, Chan and Shparlinski \cite{ChSh10}) and the fluctuations of the number of points among some families of curves (see Kurlberg and Rudnick \cite{KuRu09}, Xiong \cite{Xio10} and Bucur, David, Feigon, Lal{\'{\i}}n \cite{BDFL10a,BDFL10b}). Bounds for the number of rational points on curves in a small rectangle is crucial in the study of local spacings between fractional parts of $n^2\alpha$, see Rudnick, Sarnak and one of the authors \cite{RSZ01, Zah03}. Such questions have applications in mathematical physics, see the important works by Berry and Tabor \cite{BeTa77}, Rudnick and Sarnak \cite{RuSa98} and Sarnak \cite{Sar99}.

All the above works study analytic aspects of the number of points of families of curves over finite fields, such as bounds on the number of points and the fluctuation of the number of points along a family. In this paper we study an arithmetic property of the number of points on curves of the form
\begin{equation}\label{eqnc1}
y^{\ell}=P(x)
\end{equation}
over $\mathbb{F}_p$, when the curve is absolutely irreducible. To make it precise, we take the rational points on the curve $\CC$ as a subset in $[0,p-1]^2$, and let $\Omega\subseteq[0,p-1]^2$ be a rectangular ``window''. Instead of asking how many points are captured by $\Omega$, we ask the following question: if we move the window around the domain, what is the probability that the number of captured points is even (or odd)? This kind of problem dates back to Gauss when he proved the well-known Gauss lemma for quadratic residues, i.e. if $GCD(a,p)=1$, then if $r$ is the number of elements in the set $\{a,2a,\ldots,(\frac{p-1}{2})a\}$ that have least positive residue greater than $p/2$, then the Legendre symbol satisfies $\frac{a}{p}=(-1)^r$. Formulating in our language, this is to consider the number of points on the line $y=ax$ inside the rectangle $[1,(p-1)/2]\times(p/2,p-1]$, and then look at its residue class modulo $2$. We also note that the uniformity modulo $m$ of the values of some multiplicative functions, such as the Ramanujan tau function, was investigated by Serre \cite{Ser75}. For more results on the uniform distribution of the values of multiplicative functions modulo $m$, the reader is referred to the monograph of Narkiewicz \cite{Nar84}. Recently, Lamzouri and one of the authors \cite{LaZa11} have studied the distribution of real character sums modulo $m$.

In the present paper, given a positive integer $m$, we ask about the distribution of the number of points captured by the window $\Omega$ among each congruence class of $m$ when we move it around the domain. Since it is believed that the set of rational points on a curve exhibits a strong random behaviour, one may expect that the above mentioned probability is $1/m$. We prove that this is indeed the case when $\Omega$ has full length in the $y$-coordinate in Theorem \ref{thm1}. Next, we consider the joint distribution of the number of points on several different curves of the same form as \eqref{eqnc1}. We will see that under some natural conditions, the distributions on these different curves are independent. After that, we show that restricting the $y$-coordinate of the rectangle will retain the uniform distribution among residue classes modulo $m$. Finally, we will give an application on the distribution of $\ell$-th power residues and nonresidues in the last section.

The idea here is to relate our problems of studying the distribution of number of points modulo $m$ to that of random walks on the additive group $\mathbb{Z}/m\mathbb{Z}$. The idea is to use results on random walks showing that the distribution modulo $m$ in the random walk situation is uniform, and then show that the difference from our problem to that of the random walks can be handled, so that we get uniform distribution modulo $m$ in our context as well. For information on random walks on finite groups, the reader is referred to \cite{Hil05, Spi76}. One important feature of our result is that uniform distribution occurs already when we consider the number of points in very short intervals.

\section{Statement of Main Results}

We first fix some notations. Let $p$ be a large prime and let $\ell\geq 2$ be an integer. For a polynomial $P(x)\in\mathbb{F}_p[x]$, let $\CC$ be the curve over $\mathbb{F}_p$ defined by the equation $y^{\ell}=P(x)$. Let $I$ be a fixed positive integer (which will serve as the length of our rectangles). Define $N_{\CC}(x_0,I)$ to be the number of points on $\CC$ inside the rectangle $R_{x_0}=(x_0,x_0+I]\times[0,p-1]$, i.e.
\begin{equation*}
N_{\CC}(x_0,I) = \#\{ (x,y)\in\CC(\mathbb{F}_p) : x_0 < x \leq x_0+I \}.
\end{equation*}
Let $\CI\subseteq[0,p-1]$ be an interval, and denote $\abs{\CI}=\#(\CI\cap\mathbb{Z})$. For any $m$ with $GCD(m,\ell)=1$, we define $\Phi_{p}(P,m,a)$ to be the proportion of values $x_0\in\CI$ such that $N_{\CC}(x_0,I)\equiv a \mod{m}$, i.e.
\begin{equation*}
\Phi_{\CC}(m,a)=\frac{1}{\abs{\CI}}\#\{ 0 \leq x_0 \leq p-1 : N_{\CC}(x_0,I)\equiv a \mod{m} \}.
\end{equation*}

Our first result is that when one moves the rectangles $R_{x_0}$ along the $x$-direction, the $N_{\CC}(x_0,I)$ becomes uniformly distributed modulo $m$. Note that the distribution and the main term of the discrepancy does not depend on the lengths of the intervals $I$ and $\CI$, nor the particular position of $\CI$ as long as the conditions in the theorem are satisfied.
\begin{theorem}\label{thm1}
Let $p$ be a large prime and $P(x)\in\mathbb{F}_p[x]$ be a nonconstant polynomial of degree $d$ which is not a complete $\ell'$-th power for any $\ell'$ with $GCD(\ell',\ell)=1$. Let $L=L(p)<\frac{\log{p}}{2\log{4d}}$ be an integral function of $p$ such that $L(p)\rightarrow\infty$ as $p\rightarrow\infty$. Suppose $\CI$ is an interval such that $\CI\gg p^{\frac{1}{2}+\varepsilon}$ for some $\varepsilon>0$, and $I$ is an integer with $p-L>I>L$. Then for any positive integer $m$ with $GCD(m,\ell)=1$ we have
\begin{equation*}
\sum_{a=0}^{m-1}\left(\Phi_{\CC}(m,a)-\frac{1}{m}\right)^2 \leq \frac{7m^3\ell^2}{L(p)}+O\left( \frac{m^3 \ell^3 L(p) \sqrt{p}\log{p}}{\abs{\CI}} \right).
\end{equation*}
\end{theorem}

\begin{cor}\label{cor1}
Assumptions and notations are as in Theorem \ref{thm1}. If $m=o(L(p)^{1/5})$, then
\begin{equation*}
\Phi_{\CC}(m,a)=\frac{1}{m}+O\left( \sqrt{\frac{m^3\ell^2}{L(p)}} \right),
\end{equation*}
uniformly for all $0 \leq a \leq m-1$.
\end{cor}

\begin{remark}
Our assumption that $GCD(m,\ell)=1$ is necessary in order to obtain uniform distribution. For example, if we consider the elliptic curve $E$ defined by $y^2=x^3-n^2x$, then for each $x\neq 0,n,-n$, either there are two $y$ so that $(x,y)\in E(\mathbb{F}_p)$, or there are none. Thus $N_{E}(x_0,I)$ is almost always even, and so one cannot have uniform distribution modulo $2$. We remark that the distribution modulo $2$ in this example depends on the location of the roots of the polynomial $P(x)=x^3-n^2x$.

Although one cannot expect uniform distribution for a particular $p$ when $m$ and $\ell$ are not relatively prime, it may still be possible to have uniform distribution when we take an average over $p$. For example, let $E_p$ be the elliptic curve $y^2=x^3+x$ over $\mathbb{F}_p$, and let $m=2$. The distribution of $N_{E}(x_0,I)$ for a particular prime $p$ might not be uniform, but instead depends on the locations of the roots of $x^2+1 \mod{p}$. Now we take $N$ to be a large integer, and take an average over all primes $p\equiv 1\pmod{4}$, $p\leq N$ (here for each $p$ we normalize the points in $E_p$ by $(x,y)\mapsto(\frac{x}{p},\frac{y}{p})$, so that we have a fixed domain for all $p$). By a well-known result of Duke, Friedlander and Iwaniec \cite{DFI95}, the fractional parts $\frac{\nu}{p}$ of the roots of $x^2+1\mod{p}$ are uniformly distributed as $p$ varies. Therefore, the average values over $p\leq N$ of the number of points inside a rectangle $(x_0+I)\times[0,1)$ will be uniformly distributed modulo $2$ when $x_0$ varies.
\end{remark}

After studying the distribution of the number of points on the curve $\CC$, we continue to consider the joint distribution of the number of points on curves of the form
\begin{align*}
\CC_l: y^{\ell}=P_l(x)
\end{align*}
for $1\leq l\leq k$, where $k$ is a positive integer, and all $P_l(x)\in\mathbf{F}_p[x]$ are polynomials that are not complete $\ell$-th powers. Define
\begin{equation*}
N_{l}(x_0,I) = N_{\CC_l}(x_0,I)=\#\{ (x,y)\in\CC_l(\mathbb{F}_p : x_0 < x \leq x_0+I \},
\end{equation*}
and for any vector $\textbf{a}=(a_1,\ldots,a_k)\in\mathbb{Z}^k$,
\begin{equation*}
\Phi(m,\textbf{a})=\frac{1}{\abs{\CI}}\#\{ 0 \leq x_l \leq p-1 : N_{l}(x_l,I)\equiv a_l \mod{m} ~ \forall 1\leq l\leq k \}.
\end{equation*}
Our first observation is that various $N_{l}$'s might not be independent of each other.

\begin{remark}
For example, let $\ell=3$, $P_1(x)=x$ and $P_2(x)=x^2$, i.e.
\begin{align*}
\CC_1: & y^3=x, \\
\CC_2: & y^3=x^2.
\end{align*}
Then we claim that $N_1(x_0,I)=N_2(x_0,I)$ for any $x_0$ and $I$. Indeed, fix an $x$. If $x=0$, then both curves have a unique $y$. If $x\neq 0$ and $\CC_1$ has a point $(x,y)$, then $(x,y^2)$ is a point on $\CC_2$. Conversely, if $x\neq 0$ and $(x,y)$ is a point on $\CC_2$, then $(x,y^2/x)$ is a point on $\CC_1$. Therefore, $N_1=N_2$ as the number of points above any $x$ is the same for both curves. As an immediate consequence, for any $\textbf{a}=(a_1,a_2)$, we have
\begin{equation*}
\Phi(m,\textbf{a})=
\begin{cases}
\frac{1}{m} &, a_1=a_2, \\
0 &, a_1\neq a_2.
\end{cases}
\end{equation*}
\end{remark}

In view of the above remark, it is natural to introduce the following conditions. Let $P_1(x), \ldots, P_k(x)\in\mathbf{F}_p[x]$ be polynomials. We say that the set $\{P_1(x),\ldots,P_k(x)\}$ is \textit{multiplicatively dependent} if there exists integers (which may be positive or negative) $e_1,\ldots,e_l$ such that the combination
\begin{equation*}
Q(x)=P_1(x)^{e_1}\ldots P_k(x)^{e_k}
\end{equation*}
is identically $1$. The set of polynomials is \textit{multiplicatively independent} if it is not multiplicatively dependent.

If the polynomials are multiplicatively independent, we have the following result.
\begin{theorem}\label{thm2}
Let $k\geq 2$ be an integer. Let $p$ be a large prime and $P_1(x),\ldots,P_k(x)\in\mathbb{F}_p[x]$ be nonconstant polynomials of degree $d_1,\ldots,d_k$ respectively, which are not complete $\ell'$-th powers for any $\ell'$ with $GCD(\ell',\ell)=1$. Let $d=\max\{d_1,\ldots,d_k\}$. Suppose that the set of polynomials $\{P_1(x),\ldots,P_k(x)\}$ is multiplicatively independent. Let $L=L(p)<\frac{\log{p}}{2\log{4d}}$ be an integral function of $p$ such that $L(p)\rightarrow\infty$ as $p\rightarrow\infty$. Suppose $\CI$ is an interval such that $\CI\gg p^{\frac{1}{2}+\varepsilon}$ for some $\varepsilon>0$, and $I$ is an integer with $p-L>I>L$, then for any positive integer $m$ with $GCD(m,\ell)=1$, we have
\begin{equation*}
\sum_{\textbf{a}\in(\mathbb{Z}/m\mathbb{Z})^k}\left(\Phi(m,\textbf{a})-\frac{1}{m^k}\right)^2 \leq \frac{7m^{k+2}\ell^2}{L} +O\left(\frac{dkL\ell^3m^{k+2}\sqrt{p}\log{p}}{\abs{\CI}}\right)
\end{equation*}
\end{theorem}

An immediate corollary of the above theorem is that the $N_l(x_0,I)$ are independent. More precisely, we have the following.
\begin{cor}
Assumptions and notations are as in Theorem \ref{thm2}. If $m=o(L(p)^{1/(3k+2)})$, then
\begin{equation*}
\Phi(m,\textbf{a})=\frac{1}{m^k}+O\left( \frac{m^{k/2+1}\ell}{\sqrt{L(p)}} \right),
\end{equation*}
uniformly for all $\textbf{a}\in(\mathbb{Z}/m\mathbb{Z})^k$.
\end{cor}

So far we did not restrict the $y$-coordinates of the curves $\CC$. Our next objective is to see if a restriction of $y$-coordinates will affect the distribution of the number of points into various congruence classes. For the sake of simplicity, we only consider the case when each $x$-coordinate has at most one corresponding $y$-value in the restricted domain such that $(x,y)\in\CC$.

To be more precise, we let $\CI, \CJ \subseteq [0,p-1]$ be two intervals such that the following condition holds:
\begin{equation}\tag{$\ast$}\label{cond1}
\text{$\forall x\in\CI, \exists$~at most one $y\in\CJ$ such that $(x,y)\in\CC$.}
\end{equation}
Denote $\Omega=\CI\times\CJ$, and define
\begin{equation*}
N_{\CC,\Omega}(x_0,I) = \#\{ (x,y)\in\CC(\mathbb{F}_p)\cap\Omega : x_0 < x \leq x_0+I \},
\end{equation*}
and
\begin{equation*}
\Phi_{\CC,\Omega}(m,a)=\frac{1}{p}\#\{ 0 \leq x_0 \leq p-1 : N_{\CC,\Omega}(x_0,I)\equiv a \mod{m} \}.
\end{equation*}
Bringing into play some ideas from algebraic geometry, we prove that the numbers $N_{\CC,\Omega}(x_0,I)$ are uniformly distributed among the residue classes of $m$. Note that due to condition \eqref{cond1}, we do not need to assume that $GCD(m,\ell)=1$ in this case.
\begin{theorem}\label{thm3}
Let $p$ be a large prime and $P(x)\in\mathbb{F}_p[x]$ be a nonconstant polynomial of degree $d$ which is not a complete $\ell'$-th power for any $\ell'$ with $GCD(\ell',\ell)=1$. Let $L=L(p)=o(\log{p}/\log\log{p})$ be an integral function of $p$ such that $L(p)\rightarrow\infty$ as $p\rightarrow\infty$, and let $I$ is an integer with $p-L>I>L$ is an integer and let $\Omega=\CI\times\CJ$ be a rectangle such that condition \eqref{cond1} is satisfied, $\abs{\CJ}=\alpha p$ for some $0<\alpha\leq 1$, and $\abs{\CI}\gg p^{1/2+\delta}$ for some $\delta>0$. Then for any positive integer $m$, we have
\begin{equation*}
\sum_{a=0}^{m-1}\left(\Phi_{\CC,\Omega}(m,a)-\frac{1}{m}\right)^2 \leq \frac{4m^4}{L(p)} +O(m^4/p^{\frac{1}{2}-\varepsilon}).
\end{equation*}
for all $\varepsilon>0$.
\end{theorem}

\begin{cor}\label{cor3}
Assumptions and notations are as in Theorem \ref{thm3}. If $m=o((L(p))^{1/6})$, then
\begin{equation*}
\Phi_{\CC,\Omega}(m,a)=\frac{1}{m}+O\left( \frac{m^2}{\sqrt{L(p)}} \right),
\end{equation*}
uniformly for all $0\leq a \leq m-1$.
\end{cor}

Finally, we will apply our results above to study the distributions of power residues and nonresidues. In particular, we obtain the following result, which says that for any fixed power residue class, we can find a representative in almost all short intervals in $[0,p-1]$.
\begin{cor}\label{cor4}
Let $\ell\geq 2$ be an integer, and let $L(p)$ be an integer function of $p$ that tends to infinity as $p$ tends to infinity. For any $\ell$-th root of unity $\mu$ and for all $x_0\in[0,p-1]$ except possibly $O(p/L(p)^{1/7})$ of them, there is an $x$ inside the interval $[x_0,x_0+L(p))$ with $(\frac{x}{p})_{\ell}=\mu$, where $(\frac{\cdot}{p})_{\ell}$ denotes the $\ell$-th power residue symbol.
\end{cor}

For more results on the distribution of quadratic residues and nonresidues in short intervals, or the distribution of more general multiplicative functions in short intervals, the reader is referred to the works of Davenport and Erdos \cite{DaEr52}, Chaterjee and Soundararajan \cite{ChSo11} and Lamzouri \cite{Lam11}.

\section{Preliminaries}\label{secprelim}

In this section we collect together some preliminary results which will be used later. The first few lemmas show that certain combinations of polynomials which are not a complete $\ell'$-th powers cannot become a complete $\ell$-th power.

\begin{lemma}\label{lem21}
Let $r\geq 2$, $x_1, \ldots, x_r\in\mathbb{F}_p$ be $r$ distinct elements. Suppose $\CM$ is a nonempty finite subset of the algebraic closure $\overline{\mathbb{F}}_p$ with $4\abs{\CM}<p^{\frac{1}{r}}$.
Then there exists a $j\in\{1,\ldots,r\}$ such that the translate $\CM+x_j$ is not contained in $\cup_{i\neq j}(\CM+x_i)$.
\end{lemma}

\begin{proof}
Suppose $(x_1,\ldots,x_r,\CM)$ provides a counterexample to the statement of the lemma. Then it is clear that for any nonzero $t\in\mathbb{F}_p$, the tuple $(tx_1,\ldots,tx_r,t\CM)$ is another counterexample.

We now use Minkowski's theorem on lattice points in a convex symmetric body to find a nonzero integer $t$ such that
\begin{equation*}
\begin{cases}
\abs{t} &\leq p-1 \\
\norm{\frac{tx_1}{p}} &\leq (p-1)^{-\frac{1}{r}} \\
&\vdots \\
\norm{\frac{tx_r}{p}} &\leq (p-1)^{-\frac{1}{r}}. \\
\end{cases}
\end{equation*}
Thus there are integers $y_j$ such that
\begin{equation}\label{condyj}
\begin{cases}
\abs{y_j} &\leq p(p-1)^{-\frac{1}{r}} \\
y_j &\equiv tx_j \pmod{p} \\
\end{cases}
\end{equation}
for any $j\in\{1,\ldots,r\}$, and $(y_1,\ldots,y_r,t\CM)$ provides a counterexample. Now let $j_0$ be such that $\abs{y_{j_0}}=\max_{1\leq j\leq r}\abs{y_j}$. Choose $\alpha\in t\CM$ and consider the set $\tilde{\CM}=t\CM\cap(\alpha+\mathbb{F}_p)$. Then $(y_1,\ldots,y_r,\tilde{\CM})$ will also be a counterexample.

Note that $\alpha+\mathbb{F}_p$ can be written as a union of at most $\abs{\CM}$ intervals (i.e. subsets of $\mathbb{F}_p$ consisting of consecutive integers or its translate in $\overline{\mathbb{F}}_p$) whose endpoints are in $\tilde{\CM}$. Let $\{\alpha+a,\alpha+a+1,\ldots,\alpha+b\}$ be the longest of these intervals. Then
\[\abs{b-a} \geq \frac{p}{|\tilde{\CM}|}\geq\frac{p}{\abs{\CM}}.\]
By this, \eqref{condyj} and the hypothesis $4\abs{\CM}<p^{\frac{1}{r}}$, we have
\[\abs{b-a}>4p^{1-\frac{1}{r}}>2\abs{y_{j_0}}.\]
Now if $y_{j_0}>0$, then $\alpha+a+y_{j_0}$ belongs to $\tilde{\CM}+y_{j_0}$ but does not belong to $\cup_{i\neq j_0}(\tilde{\CM}+y_i)$, while if $y_{j_0}>0$, then $\alpha+b+y_{j_0}$ belongs to $\tilde{\CM}+y_{j_0}$ but does not belong to $\cup_{i\neq j_0}(\tilde{\CM}+y_i)$. This contradicts the fact that $(y_1,\ldots,y_r,\tilde{\CM})$ is a counterexample, and thus completes our proof.
\end{proof}


Now we are ready to prove the promised result about combinations of polynomials.
\begin{lemma}\label{lem22}
Let $\ell\geq 2$ be an integer. Let $P(x)\in\mathbb{F}_p[x]$ be a polynomial which is not a complete $\ell'$-th power for any $\ell'$ with $GCD(\ell',\ell)=1$. Let $b_1,\ldots,b_r$ be $r$ distinct elements in $\mathbb{F}_p$ with $r < (\log p)/\log (4\deg P) $.
Then for any $a\in\mathbb{F}_p$ and $\mathbf{e}=(e_1,\ldots,e_r)$ with $0\leq e_j\leq \ell-1, \mathbf{e}\neq 0$, the polynomial
\begin{equation*}
Q(x)=\prod_{j=1}^r P(ax+b_j)^{e_j}
\end{equation*}
is not a complete $\ell$-th power.
\end{lemma}
\begin{proof}
The lemma is clearly true for all $\ell$ when $r=1$. Suppose the lemma is not true, then there is a least $r>1$ (but satisfying our assumption $r < (\log p)/\log (4\deg P) $) such that a counterexample exists. Let $\tilde{\ell}$ be the least $\ell$ such that a counterexample occurs for the above $r$, then we have
\begin{equation}\label{defQx}
Q(x)=\tilde{P}(x)^{\tilde{\ell}}=\prod_{j=1}^r P(ax+b_j)^{\tilde{e}_j},
\end{equation}
where $1 \leq \tilde{e_j} < \tilde{\ell}$ (if $e_j=0$ for some $j$ we would have a smaller counterexample) and $\tilde{P}(x)\in\mathbb{F}_p[x]$.

Let $\alpha_1, \ldots, \alpha_s$ be all the \textit{distinct} zeros of $P(x)$ in $\overline{\mathbb{F}}_p$. Without loss of generality we may assume that the multiplicities $m_j$ of each $\alpha_j$ satisfy $1\leq m_j < \ell$. Clearly $1\leq s\leq \deg P$. Let $\CM=\{ a^{-1}\alpha_1,\ldots,a^{-1}\alpha_s \}$ and $x_j=-a^{-1}b_j$ for all $1 \leq j \leq r$. Note that $\CM+x_j$ is the set of zeros of $P(ax+b_j)$. Since $4\abs{\CM}=4s\leq 4\deg P < p^{\frac{1}{r}}<p^{\frac{1}{r'}}$, we can apply Lemma \ref{lem21} to obtain a $j_0$ such that at least one of the roots of $P(ax+b_{j_0})$ is distinct from the roots of all other $P(ax+b_i)$ for $i\neq j_0$. By permuting the $x_j$ and $\alpha_j$ we may assume that the above occurs for $j_0=r$, and the distinguished root is $\alpha_s$, which has multiplicity $m_s$.

If $m_s$ is relatively prime to $\tilde{\ell}$, then $\tilde{e}_r m_s$ cannot be a multiple of $\tilde{\ell}$. This means the combination $Q(x)$ cannot be a complete $\tilde{\ell}$-th power, which contradicts \eqref{defQx}. On the other hand, if $GCD(m_s,\tilde{\ell})=\frac{\tilde{\ell}}{d}>1$, then \eqref{defQx} implies that $\tilde{e}_r$ must be a multiple of $d$. Since $d<\tilde{\ell}$, we see that
\begin{equation}\label{eqnlem22a}
\frac{Q(x)}{P(ax+b_r)^{\tilde{e}_r}}=\left(\frac{\tilde{P}(x)^{\frac{\tilde{\ell}}{d}}}{P(ax+b_r)^{\frac{\tilde{e}_r}{d}}}\right)^d=\prod_{j=1}^{r-1} P(ax+b_j)^{\tilde{e}_j}
\end{equation}
is a complete $d$-th power. Thus either there exists some $\tilde{e}_j$ which is not a multiple of $d$, so \eqref{eqnlem22a} is a counterexample with smaller $r$, or each $\tilde{e}_j$ is a multiple of $d$, then
\begin{equation*}
Q(x)^{\frac{1}{d}}=\tilde{P}(x)^{\frac{\tilde{\ell}}{d}} =\prod_{j=1}^r P(ax+b_j)^{\frac{\tilde{e}_j}{d}}
\end{equation*}
is a counterexample with the same $r$ but a power smaller than $\tilde{\ell}$. In both cases we obtain a contradiction.
\end{proof}

For any positive integer $m$, denote $e_m(z)=e^{2\pi iz/m}$. Denote by $\mu_{\ell}$ the set of $\ell$-th roots of unity. For any vector $v\in\mu_{\ell}^k$, define
\begin{equation}\label{defF}
F(v)=1+v+\ldots+v^{\ell-1}=
\begin{cases}
\ell &,~v=1, \\
0 &, \text{~otherwise.}
\end{cases}
\end{equation}

We introduce the following probability model for the values of $F(v)$ based on random walks. If an $\ell$-th root of unity $v$ is drawn at random, and the probability that each root being drawn is $1/\ell$, then $F(v)=\ell$ with probability $1/\ell$ and $F(v)=0$ with probability $(\ell-1)/\ell$. Inspired by this fact, we let $\{X_j\},\{Y_j\}$ be two sequences of independent random variables so that
\begin{equation*}
P(X_j=\ell)=1/\ell \qquad \text{and} \qquad P(X_j=0)=\frac{\ell-1}{\ell},
\end{equation*}
and the same for $Y_j$. We consider the stochastic process $\{ Z_x \mod{m} \}_{x\geq 1}$, where
\begin{equation*}
Z_x = \sum_{j=1}^x X_j-\sum_{j=1}^x Y_j.
\end{equation*}
This can be viewed as a random walk on the additive group $\mathbb{Z}/m\mathbb{Z}$, with each step being the random variable $X_j-Y_j$. We are interested in the random variable
\begin{equation*}
\Phi(L;m,a)=\frac{1}{L}\abs{\{ x\leq L: Z_x\equiv a\pmod{m} \}}.
\end{equation*}
Part \eqref{prop21a} of the following proposition is in essence saying that the difference between $\Phi(L;m,a)$ and the expected value $1/m$ is not too large. Part \eqref{prop21b} of the proposition is a high dimensional version of part \eqref{prop21a}, and part \eqref{prop21c} is modeled on a slightly different situation under the same idea.

\begin{prop} \label{prop21}
Let $L$ be a positive integer.
\begin{enumerate}
\item\label{prop21a} Let $\mathbf{v}=(v_1,\ldots,v_L), \mathbf{v'}=(v_1',\ldots,v_L')\in\mu_{\ell}^k$. Suppose $GCD(\ell,m)=1$, then
\begin{equation*}
\sum_{a=0}^{m-1} \sum_{\mathbf{v},\mathbf{v}'\in\mu_{\ell}^L} \abs{\sum_{x=1}^L\sum_{t=1}^{m-1} e_m\left(t\left(\sum_{j=1}^x F(v_j)-\sum_{j=1}^x F(v_j')-a\right)\right)}^2 \leq 7m^4 L\ell^{2L+2}.
\end{equation*}
\item\label{prop21b} Let $k$ be a positive integer and $\textbf{a}=(a_1,\ldots,a_k)\in(\mathbb{Z}/m\mathbb{Z})^k$. For $1\leq l\leq k$, let $\mathbf{v}_l=(v_{l,1},\ldots,v_{l,L}), \mathbf{v}'_l=(v_{l,1}',\ldots,v_{l,L}')\in\mu_{\ell}^L$. Suppose $GCD(\ell,m)=1$, then
\begin{multline*}
\sum_{a\in(\mathbb{Z}/m\mathbb{Z})^k} \sum_{\substack{\mathbf{v}_l,\mathbf{v'}_l\in\mu_{\ell}^L \\ 1\leq l\leq k}} \abs{\sum_{x=1}^L\sum_{\mathbf{t}=(t_1,\ldots,t_k)\neq \mathbf{0}} e_m\left(\sum_{l=1}^k t_l\left(\sum_{j=1}^x F(v_{l,j})-\sum_{j=1}^x F(v_{l,j}')-a_l\right)\right)}^2 \\
\leq 7m^{2k+2} L\ell^{2Lk+2}.
\end{multline*}
\item\label{prop21c} If $\mathbf{v}=(v_1,\ldots,v_L), \mathbf{v}'=(v_1',\ldots,v_L')\in\{0,1\}^k$, then
\begin{equation*}
\sum_{a=0}^{m-1} \sum_{\mathbf{v},\mathbf{v'}\in\{0,1\}} \abs{\sum_{x=1}^L\sum_{t=1}^{m-1} e_m\left(t\left(\sum_{j=1}^x v_j-\sum_{j=1}^x v_j'-a\right)\right)}^2 \leq 2^{2L+2}m^4 L.
\end{equation*}
\end{enumerate}
\end{prop}

\begin{proof}
\eqref{prop21a} follows from \eqref{prop21b} by taking $k=1$. For \eqref{prop21b}, consider
\begin{align}
& \sum_{\mathbf{v}_l,\mathbf{v'}_l\in\mu_{\ell}^L} \abs{\sum_{x=1}^L\sum_{\mathbf{t}\neq\mathbf{0}} e_m\left(\sum_{l=1}^k t_l\left(\sum_{j=1}^x F(v_{l,j})-\sum_{j=1}^x F(v_{l,j}')-a_l\right)\right)}^2 \nonumber \\
= & \sum_{\mathbf{v}_l,\mathbf{v'}_l\in\mu_{\ell}^L} \left( \sum_{x_1=1}^L\sum_{\mathbf{t}_1\neq\mathbf{0}} e_m\left(\sum_{l=1}^k t_{l,1}\left(\sum_{j=1}^{x_1} F(v_{l,j})-\sum_{j=1}^{x_1} F(v_{l,j}')-a_l\right) \right)\right) \nonumber \\
& \qquad \times \left( \sum_{x_2=1}^L\sum_{\mathbf{t}_2\neq\mathbf{0}} e_m\left(-\sum_{l=1}^k t_{l,2}\left(\sum_{j=1}^{x_2} F(v_{l,j})-\sum_{j=1}^{x_2} F(v_{l,j}')-a\right) \right)\right) \nonumber \\
= & \sum_{\mathbf{v}_l,\mathbf{v'}_l\in\mu_{\ell}^L}\sum_{1\leq x_1,x_2\leq L}\sum_{\mathbf{t}_1,\mathbf{t}_2\neq\mathbf{0}} \prod_{l=1}^k e_m(a_l(t_{l,2}-t_{l,1})) \nonumber \\
& ~\times e_m\left( \sum_{l=1}^k \left(t_{l,1}\left( \sum_{j=1}^{x_1} F(v_{l,j})-\sum_{j=1}^{x_1} F(v_{l,j}') \right) - t_{l,2} \left( \sum_{j=1}^{x_2} F(v_{l,j})-\sum_{j=1}^{x_2} F(v_{l,j}') \right) \right) \right). \label{eqnprop21a}
\end{align}
Here $\mathbf{t}_1=(t_{1,1},\ldots,t_{k,1})$, and similarly for $\mathbf{t}_2$.

We now sum over all $a_l$ with $0\leq a_l \leq m-1$ and use the orthogonality relation
\begin{equation*}
\sum_{a_l=0}^{m-1} e_m(a_l(t_{l,2}-t_{l,1}))=
\begin{cases}
m, & t_{l,1}=t_{l,2}, \\
0, & t_{l,1}\neq t_{l,2}.
\end{cases}
\end{equation*}
Then \eqref{eqnprop21a} becomes
\begin{multline*}
m^k \sum_{\mathbf{v}_l,\mathbf{v'}_l\in\mu_{\ell}^L} \sum_{\textbf{t}\neq \textbf{0}} \sum_{1\leq x_1, x_2\leq L} \\
e_m\left( \sum_{l=1}^k t_l\left( \sum_{j=1}^{x_1} F(v_{l,j})-\sum_{j=1}^{x_1} F(v_{l,j}')- \sum_{j=1}^{x_2} F(v_{l,j})+\sum_{j=1}^{x_2} F(v_{l,j}') \right) \right).
\end{multline*}
We separate the terms with $x_1=x_2$ for which the looped sums inside the exponential vanish, which gives the total $m^k(m^k-1)L\ell^{2Lk}$. For the remaining terms, note that the looped sum for a particular pair is the negative of that of its reverse pair. So the above sum is
\begin{align}
& m^k(m^k-1)L\ell^{2Lk} + m^k \sum_{\textbf{t}\neq \textbf{0}} \sum_{1\leq x_1<x_2\leq L} \sum_{\mathbf{v}_l,\mathbf{v'}_l\in\mu_{\ell}^L} \nonumber \\
& \qquad e_m\left(\sum_{l=1}^k t_l\left( \sum_{j=x_1+1}^{x_2} F(v_{l,j})-\sum_{j=x_1+1}^{x_2} F(v_{l,j}')\right)\right) \nonumber \\
& \qquad\qquad +e_m\left( -t_l\left( \sum_{j=x_1+1}^{x_2} F(v_{l,j})-\sum_{j=x_1+1}^{x_2} F(v_{l,j}')\right) \right) \nonumber \\
= & m^k(m^k-1)L\ell^{2Lk} + 2m^k \sum_{\textbf{t}\neq\textbf{0}} \sum_{1\leq x_1<x_2\leq L} \nonumber \\
&~ \ell^{2Lk-2k(x_2-x_1)}\prod_{l=1}^k(e_m(\ell t_l)+\ell-1)^{x_2-x_1}(e_m(-\ell t_l)+\ell-1)^{x_2-x_1} \nonumber \\
= & m^k(m^k-1)L\ell^{2Lk} \nonumber \\
& \qquad + 2m^k\ell^{2Lk} \sum_{\textbf{t}\neq\textbf{0}} \sum_{1\leq x_1<x_2\leq L} \prod_{l=1}^k\left(\frac{(e_m(\ell t_l)+\ell-1)(e_m(-\ell t_l)+\ell-1)}{\ell^2}\right)^{x_2-x_1},  \label{eqnprop21b}
\end{align}
where in the penultimate step, we used
\begin{equation*}
\sum_{v^{\ell}=1}e_m(tF(v))=e_m(\ell t)+\ell -1.
\end{equation*}
For $GCD(\ell,m)=1$, we have
\begin{equation*}
\abs{\cos\left( \frac{2\pi\ell t}{m} \right)} \leq 1-\frac{\pi^2}{3m^2}
\end{equation*}
for any $1\leq t\leq m-1$. Hence,
\begin{align}
\frac{(e_m(\ell t)+\ell-1)(e_m(-\ell t)+\ell-1)}{\ell^2} &= \frac{\ell^2-2\ell+2+2(\ell-1)\cos\frac{2\pi\ell t}{m}}{\ell^2} \nonumber \\
& \leq 1-\frac{2(\ell-1)(1-\frac{\pi^2}{3m^2})}{\ell^2}. \label{eqnprop21c}
\end{align}
Fix $x_2-x_1=d$. For each $1\leq d\leq L-1$, the number of $(x_1,x_2)$ with $1\leq x_1<x_2 \leq L$ with $x_2-x_1=d$ is $L-d$. So \eqref{eqnprop21c} implies
\begin{align*}
& \sum_{1\leq x_1<x_2\leq L} \left(\frac{(e_m(\ell t)+\ell-1)(e_m(-\ell t)+\ell-1)}{\ell^2}\right)^{x_2-x_1} \\
\leq & \sum_{d=1}^{L-1} (L-d)\left( 1-\frac{2(\ell-1)(1-\frac{\pi^2}{3m^2})}{\ell^2} \right)^{d} \\
\leq & 3m^2\ell^2 L
\end{align*}
after some simplification. For any $\textbf{t}\neq\textbf{0}$ we have a nonzero coordinate for which the above calculations apply. Thus
\begin{equation*}
\sum_{\textbf{t}\neq\textbf{0}} \sum_{1\leq x_1<x_2\leq L} \prod_{l=1}^k\left(\frac{(e_m(\ell t_l)+\ell-1)(e_m(-\ell t_l)+\ell-1)}{\ell^2}\right)^{x_2-x_1} \leq (m^k-1)(3m^2\ell^2L).
\end{equation*}
Part \eqref{prop21b} now follows easily by inserting the above estimate in \eqref{eqnprop21b}.

For \eqref{prop21c}, we derive as above that
\begin{multline}\label{eqnprop21d}
\sum_{\mathbf{v},\mathbf{v'}\in\mu_{\ell}^L} \abs{\sum_{x=1}^L\sum_{t=1}^{m-1} e_m\left(t\left(\sum_{j=1}^x v_j-\sum_{j=1}^x v_j' -a\right)\right)}^2 \\
= m(m-1)2^{2L}L + m \sum_{t=1}^{m-1} \sum_{1\leq x_1<x_2\leq L} \sum_{\mathbf{v},\mathbf{v'}\in\{0,1\}} e_m\left( t\left( \sum_{j=x_1+1}^{x_2} v_j-\sum_{j=x_1+1}^{x_2} v_j'\right)\right) \\
+e_m\left( -t\left( \sum_{j=x_1+1}^{x_2} v_j-\sum_{j=x_1+1}^{x_2} v_j'\right) \right).
\end{multline}
Here from
\begin{equation*}
\sum_{v\in\{0,1\}}e_m(tv)=1+e_m(t)
\end{equation*}
and the inequality
\begin{equation*}
\abs{\cos\left( \cos\frac{\pi t}{m} \right)} \leq 1-\frac{\pi^2}{3m^2},
\end{equation*}
we see that the second term in \eqref{eqnprop21d} is
\begin{align*}
& 2\cdot2^{2L}m\sum_{t=1}^{m-1}\sum_{1\leq x_1<x_2\leq L} \left(\frac{(1+e_m(t))(1+e_m(-\ell t))}{4}\right)^{x_2-x_1} \\
= & 2^{2L+1}m^2 \sum_{d=1}^{L-1} (L-d)\left( \cos\frac{\pi t}{m} \right)^{d} \\
\leq & 2^{2L+1}m^2L \sum_{d=1}^{L-1} \left( 1-\frac{\pi^2}{3m^2} \right)^{d} \\
= & 2^{2L+1}m^2L \left( 1-\frac{\pi^2}{3m^2} \right) \frac{1-\left( 1-\frac{\pi^2}{3m^2} \right)^{2L-2}}{\frac{\pi^2}{3m^2}} \\
\leq & 2^{2L+1}m^4 L.
\end{align*}
Substituting this back into \eqref{eqnprop21d} completes the proof of \eqref{prop21c}.

\end{proof}

The next lemma is the classical Weil bound for incomplete exponential sums over $\mathbb{F}_p$. Let $\chi_{\ell}$ be a nontrivial multiplicatively character of order $\ell$. For a polynomial $P(x)\in\mathbb{F}_p[x]$ of degree $d$ and an interval $\CI\subseteq[0,p-1]$, define
\begin{equation*}
S_{\CI}(P)=\sum_{x\in\CI}\chi_{\ell}(P(x)).
\end{equation*}

\begin{lemma}\label{lemweil}
If $P(x)$ is not a complete $\ell$-th power, then
\begin{equation*}
\abs{S_{\CI}(P)}\leq 2(d+1)\sqrt{p}\log{p}.
\end{equation*}
\end{lemma}
\begin{proof}
If $\CI$ is the complete interval $[0,p-1]$, the result follows from Weil's estimate \cite{Wei48a}. The same estimate hold for the sum:
\begin{equation}\label{eqnlemweila}
\abs{\sum_{x\in[0,p-1]} \chi_{\ell}(P(x))e_p(-tx)} \leq (d+1) \sqrt{p}
\end{equation}
for any $t\in\mathbb{F}_p$. If $\CI$ is not the complete interval, let $\CI\cap\mathbb{Z}=\{ a, a+1,\ldots,b \}$. We use a standard method to express the incomplete sum $S_{\CI}(P)$ in terms of complete sums. More precisely, we have
\begin{equation*}
S_{\CI}(P)=\sum_{x\in[0,p-1]}\chi_{\ell}(P(x))\left( \frac{1}{p}\sum_{n\in\CI}\sum_{t\mod{p}}e_p(t(n-x)) \right).
\end{equation*}
Changing the order of summation and using \eqref{eqnlemweila}, we get
\begin{align}
\abs{S_{\CI}(P)} &= \abs{\frac{1}{p}\sum_{t\text{~mod~$p$}}\left(\sum_{n\in\CI}e_p(tn)\right)\left(\sum_{x\in[0,p-1]}\chi_{\ell}(P(x))e_p(-tx)\right)} \nonumber \\
&\leq \frac{1}{p}(d+1)\sqrt{p}\abs{\sum_{t\text{~mod~$p$}}\left(\sum_{n\in\CI}e_p(tn)\right)} \nonumber \\
&= \frac{1}{p}(d+1)\sqrt{p}\left(\abs{\CI}+\abs{\sum_{t\neq 0\text{~mod~$p$}}\frac{e_p(t(a+1))-e_p(t(b+1))}{1-e_p(t)}}\right) \nonumber \\
&= \frac{1}{p}(d+1)\sqrt{p}\left(\abs{\CI}+\sum_{t\neq 0\text{~mod~$p$}}\frac{1}{\abs{\sin(t\pi/p)}}\right). \label{eqnlem22b}
\end{align}
Since $\abs{\sin(t\pi/p)}\geq \frac{\pi\abs{t}}{2p}$, we obtain
\begin{equation*}
\sum_{t\neq 0\text{~mod~$p$}}\frac{1}{\abs{\sin(t\pi/p)}} \leq 2\sum_{t=1}^{\frac{p-1}{2}}\frac{2p}{\pi \abs{t}} \leq \frac{4}{\pi}p\log{p}.
\end{equation*}
Inserting the above estimate into \eqref{eqnlem22b}, we obtain
\begin{equation*}
\abs{S_{\CI}(P)} \leq \frac{1}{p}(d+1)\sqrt{p}(\abs{\CI}+\frac{4}{\pi}p\log{p}) \leq 2(d+1)\sqrt{p}\log{p}.
\end{equation*}
This finishes the proof of the lemma.
\end{proof}

\section{Distribution of the number of points in residue classes: proof of Theorem \ref{thm1}}

Recall that we are studying the curve
\begin{equation*}
\CC: \qquad y^{\ell}=P(x).
\end{equation*}
We defined the quantities
\begin{equation*}
N_{\CC}(x_0,I) = \#\{ (x,y)\in\CC(\mathbb{F}_p) : x_0 < x \leq x_0+I \},
\end{equation*}
which is the number of points on $\CC$ inside a rectangle of some fixed length $I$, and
\begin{equation*}
\Phi_{\CC}(m,a)=\frac{1}{\abs{\CI}}\#\{ 0 \leq x_0 \leq p-1 : N_{\CC}(x_0,I)\equiv a \mod{m} \},
\end{equation*}
which can be regarded as the probability of the occurrence of $N_{\CC}(x_0,I)\equiv a \mod{m}$ for $x_0\in\CI$.

Let $N$ be a large number, $x_1,\ldots,x_r \in\mathbb{F}_p$ be distinct points and let $\mathbf{x}=(x_1,\ldots,x_r)$. Let $P(x)\in\mathbf{F}_p$ be a polynomial of degree $d$, and $\mathbf{v}=(v_1,\ldots,v_r)\in\mu_{\ell}^r$. Suppose $L\neq 0$ is an integer, and define
\begin{equation}\label{defMPv}
M_{P}(\mathbf{v})=M_{P,r,N,k}(\mathbf{v},\mathbf{x})=\{ 0\leq i \leq N: \chi_{\ell}(P(iL+x_j))=v_j ~\forall 1\leq j \leq r \}.
\end{equation}
This will serve as our bridge between the character values and the random walk setting. The following proposition estimates the size of $M_{P}(\mathbf{v})$.

\begin{prop}\label{prop22}
If $r < (\log p)/\log (4d)$ and $P(x)$ is not a complete $\ell$-th power, then for any $\mathbf{v}\in\mu_{\ell}^r$, we have
\begin{equation*}
\#M_{P}(\mathbf{v}) = \frac{N}{\ell^r}+\frac{2(dr(\ell-1)+1)}{\ell^r}\sqrt{p}\log{p}+O(d).
\end{equation*}
\end{prop}
\begin{proof}
The number of points $x\in\mathbb{F}_p$ with $P(x)=0$ is $O(\deg P)=O(d)$. Hence, there are $N+O(d)$ indices $i$ such that $P(iL+x_j)\neq 0$ for all $1\leq j\leq r$. For those $i$, we have
\begin{equation*}
\frac{1}{\ell^r}\prod_{j=1}^r F(v_j^{-1}\chi_{\ell}(P(iL+x_j)))=
\begin{cases}
1 &, \text{~if~}i\in M_{P}(\mathbf{v}), \\
0 &, \text{~otherwise},
\end{cases}
\end{equation*}
where $F(v)$ is defined in \eqref{defF}. Thus,
\begin{equation*}
\#M_{P}(\mathbf{v}) = \frac{1}{\ell^r} \sum_{i=0}^N \prod_{j=1}^r F(v_j^{-1}\chi_{\ell}(P(iL+x_j))) +O(d).
\end{equation*}
Expanding the above product and changing the order of summation, we obtain
\begin{multline}
\#M_{P}(\mathbf{v}) = \frac{N}{\ell^{r}}+\frac{1}{\ell^r}\sum_{\substack{\mathbf{e}=(e_1,\ldots,e_r)\neq 0 \\ 0\leq e_j \leq \ell-1}} v_1^{-e_1}\ldots v_r^{-e_r} \\
\times \sum_{i=0}^N\chi_{\ell}(P(iL+x_1))^{e_1}\ldots\chi_{\ell}(P(iL+x_r))^{e_r} + O(d). \label{eqnprop22a}
\end{multline}
Since the $x_j$ are distinct points on $\mathbb{F}_p$ and $r < (\log p)/\log (4\deg P) $, Lemma \ref{lem22} shows that the polynomial
\begin{equation*}
Q(i)=P(iL+x_1)^{e_1}\ldots P(iL+x_r)^{e_r}
\end{equation*}
is not a complete $\ell$-th power. Hence, by Lemma \ref{lemweil}, we have
\begin{equation*}
\abs{\sum_{i=1}^N\chi_{\ell}\left(P(iL+x_1)^{e_1}\ldots P(iL+x_r)^{e_r}\right)} \leq 2(dr(\ell-1)+1)\sqrt{p}\log{p}.
\end{equation*}
Inserting the above estimate back into \eqref{eqnprop22a}, we obtain
\begin{equation*}
\#M_{P}(\mathbf{v}) \leq \frac{N}{\ell^r}+\frac{2(dr(\ell-1)+1)}{\ell^r}\sqrt{p}\log{p}+O(d).
\end{equation*}
\end{proof}

We are now ready to prove Theorem \ref{thm1}.
\begin{proof}[Proof of Theorem \ref{thm1}]
Let $L=L(p)\leq\left[\frac{\log{p}}{2\log{4d}}\right]$ be a large number, and let $N=[\abs{\CI}/L]-1$. Define
\begin{equation*}
R_{P,m,a}(i,L)=\#\{ 1\leq x \leq L: N_{\CC}(iL+x,I)\equiv a \pmod{m} \}.
\end{equation*}
We have
\begin{equation}\label{eqnthm1a}
\abs{\Phi_{\CC}(m,a)-\frac{1}{m}}\leq\frac{1}{\abs{\CI}}\sum_{i=0}^N\abs{R_{P,m,a}(i,L)-\frac{L}{m}}+O\left(\frac{L}{\abs{\CI}}\right).
\end{equation}
By the Cauchy-Schwarz inequality,
\begin{equation*}
\left(\sum_{i=0}^N\abs{R_{P,m,a}(i,L)-\frac{L}{m}}\right)^2 \leq (N+1)\sum_{i=0}^N \left(R_{P,m,a}(i,L)-\frac{L}{m}\right)^2.
\end{equation*}
Putting this back into \eqref{eqnthm1a}, we obtain
\begin{equation}\label{eqnthm1b}
\left(\Phi_{\CC}(m,a)-\frac{1}{m}\right)^2 \leq \frac{N+1}{\abs{\CI}^2}\sum_{i=0}^N \left(R_{P,m,a}(i,L)-\frac{L}{m}\right)^2 + O\left(\frac{L^2}{\abs{\CI}^2}\right).
\end{equation}
Now note that
\begin{equation*}
N_{\CC}(iL+x,I)=N_{\CC}(iL,I)+\sum_{j=1}^{x}F(\chi_{\ell}(P(iL+I+j)))-\sum_{j=1}^{x}F(\chi_{\ell}(P(iL+j))),
\end{equation*}
so if we set
\begin{multline*}
R'_{P,m,b}(i,L) = \frac{1}{m}\#\{ 1\leq x\leq L: \\
\sum_{j=1}^{x}F(\chi_{\ell}(P(iL+I+j)))-\sum_{j=1}^{x}F(\chi_{\ell}(P(iL+j))) \equiv b \pmod{m} \}
\end{multline*}
and use the substitution $b=a+N_{\CC}(sL,I)$, then
\begin{equation}\label{eqnthm1c}
\sum_{a=0}^{m-1}\sum_{i=0}^N \left(R_{P,m,a}(i,L)-\frac{L}{m}\right)^2 = \sum_{b=0}^{m-1}\sum_{i=0}^N \left(R'_{P,m,b}(i,L)-\frac{L}{m}\right)^2.
\end{equation}
Using the orthogonality of character sums, we get
\begin{multline*}
R'_{P,m,b}(i,L) \\
=\sum_{x=1}^L\sum_{t=0}^{m-1} e_m\left(t\left(\sum_{i=1}^x F(\chi_{\ell}(P(iL+I+j))-\sum_{i=1}^x F(\chi_{\ell}(P(iL+j))-b\right)\right),
\end{multline*}
and hence
\begin{align}
& \sum_{i=0}^N \left(R'_{P,m,b}(i,L)-\frac{L}{m}\right)^2 \nonumber \\
=& \frac{1}{m^2}\sum_{i=0}^N \abs{\sum_{x=1}^L\sum_{t=0}^{m-1} e_m\left(t\left(\sum_{j=1}^x F(\chi_{\ell}(P(iL+I+j))-\sum_{j=1}^x F(\chi_{\ell}(P(iL+j))-b\right)-L\right)}^2 \nonumber \\
=& \frac{1}{m^2}\sum_{i=0}^N \abs{\sum_{x=1}^L\sum_{t=1}^{m-1} e_m\left(t\left(\sum_{j=1}^x F(\chi_{\ell}(P(iL+I+j))-\sum_{j=1}^x F(\chi_{\ell}(P(iL+j))-b\right)\right)}^2 \nonumber \\
=& \frac{1}{m^2}\sum_{\textbf{v},\textbf{v}'\in\mu_{\ell}^{L}} \abs{\sum_{x=1}^L\sum_{t=1}^{m-1} e_m\left(t\left(\sum_{j=1}^x F(v_j)-\sum_{j=1}^x F(v_j')-b\right)\right)}^2\cdot \#M_P(\textbf{w}), \label{eqnthm1d}
\end{align}
where $M_P(\textbf{w})$ is defined in \eqref{defMPv}, with $\mathbf{v}=(v_1,\ldots,v_L)$, $\mathbf{v}'=(v_1',\ldots,v_L')$, $\mathbf{w}=(v_1,\ldots,v_L,v_1',\ldots,v_L')$, and
\begin{equation*}
\mathbf{x}=(iL+I+1,\ldots,iL+I+L,iL+1,\ldots,iL+L).
\end{equation*}
Note that as $p-L > I > L$, the entries in $\mathbf{x}$ are indeed distinct. Putting \eqref{eqnthm1d} back into \eqref{eqnthm1c}, applying Proposition \ref{prop21}\eqref{prop21a} and Proposition \ref{prop22}, we have after some simplifications
\begin{equation*}
\sum_{a=0}^{m-1}\sum_{i=0}^N \left(R_{P,m,a}(i,L)-\frac{L}{m}\right)^2 = \frac{7m^3\ell^2LN^2}{\abs{\CI}^2}+O\left(\frac{m^3\ell^3 L^2 N\sqrt{p}\log{p}}{\abs{\CI}^2}\right).
\end{equation*}
Combining this estimate with \eqref{eqnthm1b}, we obtain
\begin{align*}
\sum_{a=0}^{m-1}\left(\Phi_{\CC}(m,a)-\frac{1}{m}\right)^2 & \leq \frac{7m^3\ell^2}{L}+O\left( \frac{m^3 \ell^3 L \sqrt{p}\log{p}}{\abs{\CI}} \right).
\end{align*}
This completes the proof of Theorem \ref{thm1}.
\end{proof}

\section{Joint distribution among curves: proof of Theorem \ref{thm2}}

Before we prove Theorem \ref{thm2}, we need a generalization of Proposition \ref{prop22}. Let $x_1,\ldots,x_r \in\mathbb{F}_p$ be distinct points, and let $\mathbf{x}=(x_1,\ldots,x_r)$. For $1\leq l\leq k$, let $P_l(x)\in\mathbf{F}_p$ be polynomials of degree $d_l$, $d=d_1+\ldots+d_k$, and $\mathbf{v}_l=(v_{l,1},\ldots,v_{l,r})\in\mu_{\ell}^r$. Suppose $L\neq 0$ is an integer, and define the set
\begin{equation*}
M_{P_1,\ldots,P_k}(\mathbf{v}_1,\ldots,\mathbf{v}_k)=\{ 0\leq i \leq N: \chi_{\ell}(P_l(iL+x_j))=v_{l,j} ~\forall 1\leq j \leq r, 1\leq l\leq k \}.
\end{equation*}

\begin{prop}\label{prop23}
Assume the $P_l(x)$ are not complete $\ell$-th powers, and the set $\{P_1(x),\ldots,P_k(x)\}$ is multiplicatively independent. If $r < (\log p)/\log (4d)$,
then for any $\mathbf{v}_1,\ldots,\mathbf{v}_k$, we have
\begin{equation*}
\#M_{P_1,\ldots,P_k}(\mathbf{v}_1,\ldots,\mathbf{v}_k) = \frac{N}{\ell^{kr}}+\frac{2dkr(\ell-1)+1}{\ell^{kr}}\sqrt{p}\log{p}+O(d).
\end{equation*}
\end{prop}

\begin{proof}
We follow the same idea as in the proof of Proposition \ref{prop22}. For those $x\in\mathbb{F}_p$ which are not roots of any $P_l$, we have
\begin{equation*}
\frac{1}{\ell^{kr}}\prod_{l=1}^k\prod_{j=1}^r F(v_j^{-1}\chi_{\ell}(P_l(iL+x_j)))=
\begin{cases}
1 &, \text{~if~}i\in M_{P_1,\ldots,P_k}(\mathbf{v}_1,\ldots,\mathbf{v}_k), \\
0 &, \text{~otherwise}.
\end{cases}
\end{equation*}
So,
\begin{equation*}
\#M_{P_1,\ldots,P_k}(\mathbf{v}_1,\ldots,\mathbf{v}_k) = \frac{1}{\ell^{kr}} \sum_{i=0}^N \prod_{l=1}^k\prod_{j=1}^r F(v_j^{-1}\chi_{\ell}(P_l(iL+x_j))) +O(d).
\end{equation*}
Expanding the above product, we obtain
\begin{multline}
\#M_{P_1,\ldots,P_k}(\mathbf{v}_1,\ldots,\mathbf{v}_k) = \frac{N}{\ell^{kr}}+\frac{1}{\ell^{kr}}\sum_{\mathbf{e}\in S} \prod_{l=1}^k\prod_{j=1}^r v_{l,j}^{-e_{l,j}} \\
\times \sum_{i=0}^N\chi_{\ell}\left(\prod_{l=1}^k\prod_{j=1}^r P_l(iL+x_r)\right)^{e_{l,r}} + O(d), \label{eqnprop23a}
\end{multline}
where
\begin{equation*}
S=\{\mathbf{e}=(e_{l,j})_{\substack{1\leq l\leq k \\ 1\leq j\leq r}}: 0\leq e_{l,j} \leq \ell-1\}.
\end{equation*}
As $r < (\log p)/\log (4d)$ and the $P_l$'s are multiplicatively independent, Lemma \ref{lem22} implies that the polynomial
\begin{equation*}
Q(i)=\prod_{l=1}^k\prod_{j=1}^r P_l(iL+x_r)^{e_{l,r}}
\end{equation*}
cannot be a complete $\ell$-th power for any choice of $\mathbf{e}\in S$ unless $\mathbf{e}$ is the zero vector. Therefore, we can employ Lemma \ref{lemweil} in \eqref{eqnprop23a} to get
\begin{equation*}
\#M_{P_1,\ldots,P_k}(\mathbf{v}_1,\ldots,\mathbf{v}_k) = \frac{N}{\ell^{kr}}+\frac{2(dkr(\ell-1)+1)}{\ell^{kr}}\sqrt{p}\log{p}+O(d).
\end{equation*}
\end{proof}

\begin{proof}[Proof of Theorem \ref{thm2}]
The proof of Theorem \ref{thm2} follows the same line as that of Theorem \ref{thm1}. Let $L=L(p)\leq\left[\frac{\log{p}}{2\log{4d}}\right]$, and let $N=[\abs{\CI}/L]-1$. Define
\begin{equation*}
R_{m,\textbf{a},k}(i,L)=\#\{ 1\leq x \leq L: N_{l}(iL+x,I)\equiv a_l \pmod{m} ~ \forall 1\leq l\leq k \}.
\end{equation*}
We have
\begin{equation*}
\abs{\Phi(m,\textbf{a})-\frac{1}{m^k}}\leq\frac{1}{\abs{\CI}}\sum_{i=0}^N\abs{R_{m,\textbf{a},k}(i,L)-\frac{L}{m^k}}+O\left(\frac{L}{p}\right).
\end{equation*}
Again by Cauchy-Schwarz inequality,
\begin{equation*}
\left(\sum_{i=0}^N\abs{R_{m,\textbf{a},k}(i,L)-\frac{L}{m^k}}\right)^2 \leq (N+1)\sum_{i=0}^N \left(R_{m,\textbf{a},k}(i,L)-\frac{L}{m^k}\right)^2,
\end{equation*}
which implies
\begin{equation*}
\abs{\Phi(m,\textbf{a})-\frac{1}{m^k}}^2 \leq \frac{N+1}{\abs{\CI}^{2}}\sum_{i=0}^N \left(R_{m,\textbf{a},k}(i,L)-\frac{L}{m^k}\right)^2+O\left(\frac{L^2}{p^{2}}\right).
\end{equation*}
Note that
\begin{equation*}
N_{l}(iL+x,I)=N_{l}(iL,I)+\sum_{j=1}^{x}F(\chi_{\ell}(P_l(iL+I+j)))-\sum_{j=1}^{x}F(\chi_{\ell}(P_l(iL+j)))
\end{equation*}
for all $1\leq l\leq k$. To simplify the notations, write
\begin{equation*}
\Sigma(x,l)=\sum_{j=1}^{x}F(\chi_{\ell}(P_l(iL+I+j)))-\sum_{j=1}^{x}F(\chi_{\ell}(P_l(iL+j))).
\end{equation*}
Let $\textbf{b}=\textbf{a}+(N_{l}(sL,I))_{1\leq l\leq k}$, and set
\begin{equation*}
R'_{m,\textbf{b}}(i,L) =\#\{ 1\leq x\leq L: \Sigma(x,l) \equiv b_l \pmod{m} ~ \forall 1\leq l\leq k \},
\end{equation*}
then
\begin{equation}\label{eqnthm2a}
\sum_{a\in(\mathbb{Z}/m\mathbb{Z})^k}\sum_{i=0}^N \left(R_{m,\textbf{a}}(i,L)-\frac{L}{m^k}\right)^2 = \sum_{b\in(\mathbb{Z}/m\mathbb{Z})^k}\sum_{i=0}^N \left(R'_{m,\textbf{b}}(i,L)-\frac{L}{m^k}\right)^2.
\end{equation}
Since
\begin{equation*}
R'_{m,\mathbf{b}}(i,L)=\frac{1}{m^k} \sum_{x=1}^L \prod_{l=1}^k\sum_{t_l=0}^{m-1} e_m(t_l\Sigma(x,l)-b_l),
\end{equation*}
a similar calculation as in \eqref{eqnthm1d} gives
\begin{align}
& \sum_{i=0}^N \left(R'_{m,\mathbf{b}}(i,L)-\frac{L}{m^k}\right)^2 \nonumber \\
= & \frac{1}{m^{2k}}\sum_{i=0}^N \abs{\sum_{x=1}^L \prod_{l=1}^k\sum_{t_l=0}^{m-1} e_m(t_l\Sigma(x,l)-b_l)-L}^2 \nonumber \\
= & \frac{1}{m^{2k}}\sum_{\mathbf{v}_l,\mathbf{v}_l'\in\mu_{\ell}^{L}} \abs{\sum_{x=1}^L\sum_{\mathbf{t}\neq \mathbf{0}} e_m\left(\sum_{l=1}^k t_l\left(\sum_{j=1}^x F(v_{l,j})-\sum_{j=1}^x F(v_{l,j}')-b_l\right)\right)}^2 \label{eqnthm2b} \\
& \qquad \times \#M_{P_1,\ldots,P_k}(\mathbf{w}_1,\ldots,\mathbf{w}_k), \nonumber
\end{align}
with $\textbf{w}_l=(\textbf{v}_l,\textbf{v}_l')$. Substituting \eqref{eqnthm2b} back into \eqref{eqnthm2a}, applying Proposition \ref{prop21}\eqref{prop21b} and Proposition \ref{prop23}, we obtain
\begin{equation*}
\sum_{\textbf{a}\in(\mathbb{Z}/m\mathbb{Z})^k}\sum_{i=0}^N \left(R'_{m,\textbf{b}}(i,L)-\frac{L}{m^k}\right)^2\leq 7NLm^{k+2}\ell^2+O(dkL^2\ell^3m^{k+2}\sqrt{p}\log{p}),
\end{equation*}
and so
\begin{equation*}
\sum_{\textbf{a}\in(\mathbb{Z}/m\mathbb{Z})^k}\left(\Phi(m,\textbf{a})-\frac{1}{m^k}\right)^2 \leq \frac{7m^{k+2}\ell^2}{L}+O\left(\frac{dkL\ell^3m^{k+2}\sqrt{p}\log{p}}{\abs{\CI}}\right).
\end{equation*}
\end{proof}

\section{The case of restricted domains: proof of Theorem \ref{thm3}}

In this section we study the case when the domain is restricted to a smaller rectangle $\Omega=\CI\times\CJ$ that satisfies the condition \eqref{cond1}. For any $x\in[0,p-1]$, define
\begin{equation*}
\delta_{\CC,\Omega}(x) =
\begin{cases}
1 &, \text{~if~} x\in\CI \text{~and~} \exists y\in\CJ \text{~so that~} (x,y)\in\CC, \\
0 &, \text{~otherwise.}
\end{cases}
\end{equation*}
Let $\mathbf{x}=(x_1,\ldots,x_r)\in[0,p-1]^r$, and let $\mathbf{v}=(v_1,\ldots,v_r)\in\{0,1\}^r$ be a vector. As in the proofs of previous theorems, we introduce a set and estimate its size. Define
\begin{equation*}
M_{\CC,\Omega}(\mathbf{v})=\{ x\in\CI: L|x, \delta_{\CC,\Omega}(x+x_j)=v_j ~\forall 1\leq j \leq r \}.
\end{equation*}

\begin{remark}
For $x\in\CI$, one can write down an explicit formula for $\delta_{\CC,\Omega}(x)$ involving exponential sums. Write the defining equation of $\CC$ as $f(x,y):=y^{\ell}-P(x)=0$. Consider
\begin{equation*}
S(x) = \sum_{y\in\CJ}\sum_{t\in\mathbb{F}_p} t f(x,y).
\end{equation*}
Then $S$ is the number of points in $\CC\cap\Omega$. Now our assumption \eqref{cond1} guarantees that $\delta_{\CC,\Omega}(x)=S(x)$. This formula was used by Dwork \cite{Dwo64} to prove the rationality of zeta functions of varieties over finite fields. We will not need this formula in our paper.
\end{remark}

In previous sections, we used characters to relate the random walk setting and the distribution of number of points on $\CC$, which does not allow us to control the $y$-coordinates. To allow restrictions on the domain, we proceed as follows. Let $\CH=\{h_1,\ldots,h_r\}\subseteq[0,p-1]$ be a set of integers. From the curve $\CC$ defined by \eqref{eqnc1}, we construct the $x$-\textit{shifted curve} $\CC_{\CH}$ to be the curve defined by the following system of equations:
\begin{align*}
y_1^{\ell}&=P(x+h_1) \\
y_2^{\ell}&=P(x+h_2) \\
&\vdots \\
y_r^{\ell}&=P(x+h_r).
\end{align*}
It is easy to see that $\CC_{\CH}$ is indeed a curve. The next lemma shows that this curve is absolutely irreducible if $r$ is not too large.

\begin{lemma}
If $r < \frac{\log p}{\log (4d)}$, then $\CC_{\CH}$ is absolutely irreducible.
\end{lemma}
\begin{proof}
It suffices to show that for any $\textbf{e}=(e_1,\ldots,e_r)$ with $0\leq e_j\leq \ell-1$, $\mathbf{e}\neq 0$, the combination
\begin{equation*}
Q(x)=\prod_{j=1}^r P(x+h_j)^{e_j}
\end{equation*}
cannot be a complete $\ell$-th power, and this is shown in Lemma \ref{lem22}.
\end{proof}

Let $\Omega=\CI\times\CJ\subseteq[0,p-1]^2$ be a rectangle, and let $N_{\CC,\Omega}(\CH)$ be the number of points on $\CC_{\CH}$ inside $\Omega$ with $L|x$. Since $\CC_{\CH}$ is absolutely irreducible, we can determine $N_{\CC,\Omega}(\CH)$ using the idea of generalized Lehmer problem on curves \cite{CoZa01}. In particular, we have
\begin{equation*}
N_{\CC,\Omega}(\CH) = \frac{\abs{\CI}}{L}\cdot\frac{\abs{\CJ}^{\abs{\CH}}}{p^{\abs{\CH}}}+O(\sqrt{p}\log^{\abs{\CH}+1}p),
\end{equation*}
where $\abs{\CI}=\#(\CI\cap\mathbb{Z})$. Note that $N_{\CC,\Omega}(\CH)$ only depends on the cardinality of $\CH$ but not the particular elements in it. Suppose now $\Omega$ satisfies \eqref{cond1}, then it is easy to see that
\begin{equation*}
N_{\CC,\Omega}(\CH) = \sum_{x\in\CI, L|x}\prod_{h\in\CH}\delta_{\CC,\Omega}(x+h).
\end{equation*}
Thus, if
\begin{equation}\label{eqn61}
\#\{ x+x_r : x\equiv 0\pmod{L}, x+x_r\notin \CI \} = O(\sqrt{p}),
\end{equation}
then we can estimate $M_{\CC,\Omega}(\mathbf{v})$ using $N_{\CC,\Omega}(\CH)$ by a combinatorial argument as follows. Divide the $x_j$'s into two disjoint sets according to the corresponding values of $v_j$, say
\begin{equation}\label{eqnAB}
\CA=\{ x_j : v_j=1 \} \qquad \text{and} \qquad \CB=\{ x_l : v_l=0 \}.
\end{equation}
Then
\begin{align*}
\#M_{\CC,\Omega}(\mathbf{v}) &= \sum_{x\in\CI, L|x} \prod_{x_j\in\CA} \delta(x+x_j) \prod_{x_l\in\CB} (1-\delta(x+x_l)) +O(\sqrt{p}) \\
&= \sum_{x\in\CI, L|x}\prod_{x_j\in\CA} \delta(x+x_j) \sum_{\CE\subset\CB}(-1)^{\abs{\CE}}\prod_{x_l\in\CE}\delta(x+x_l) +O(\sqrt{p})\\
&= \sum_{\CE\subset\CB}(-1)^{\abs{\CE}}\sum_{x\in\CI, L|x}\prod_{x_j\in\CA\cup\CE}\delta(x+x_j) +O(\sqrt{p})\\
&= \sum_{\CE\subset\CB}(-1)^{\abs{\CE}}N_{\CC,\Omega}(\CA\cup\CE) +O(\sqrt{p}) \\
&= \sum_{\CE\subset\CB}(-1)^{\abs{\CE}} \frac{\abs{\CI}}{L}\cdot\frac{\abs{\CJ}^{\abs{\CA}+\abs{\CE}}}{p^{\abs{\CA}+\abs{\CE}}} +O(\sqrt{p}\log^{\abs{\CA}+\abs{\CE}+1}p) \\
&= \frac{\abs{\CI}}{L}\cdot\left(\frac{\abs{\CJ}}{p}\right)^{\abs{\CA}} \left( 1-\frac{\abs{\CJ}}{p} \right)^{\abs{\CB}} +O(2^r\sqrt{p}\log^{r+1}p).
\end{align*}

We summarize the above results in the following proposition.
\begin{prop}\label{prop24}
Suppose the domain $\CI\times\CJ$ satisfies \eqref{cond1}, and let $\mathbf{x}=(x_1,\ldots,x_r)\in[0,p-1]^r$ such that \eqref{eqn61} is satisfied. Then for any $\mathbf{v}_1,\ldots,\mathbf{v}_k$, we have
\begin{equation*}
\#M_{\CC,\Omega}(\mathbf{v}) = \frac{\abs{\CI}}{L}\cdot\left(\frac{\abs{\CJ}}{p}\right)^{\abs{\CA}} \left( 1-\frac{\abs{\CJ}}{p} \right)^{\abs{\CB}} +O(2^r\sqrt{p}\log^{r+1}p),
\end{equation*}
where $\CA$ and $\CB$ is as in \eqref{eqnAB}. In particular,
\begin{equation*}
\#M_{\CC,\Omega}(\mathbf{v}) \leq \frac{\abs{\CI}}{2^r L}+O(2^r\sqrt{p}\log^{r+1}p).
\end{equation*}
\end{prop}
Note that the above proposition is meaningful only when the main term is larger than the error term, hence we need the conditions on $\CI$, $\CJ$ as stated in Theorem \ref{thm3}. We are now ready to prove Theorem \ref{thm3}.

\begin{proof}[Proof of Theorem \ref{thm3}]
Let $L$ be a large integer of order $o(\log{p}/\log\log{p})$, and $N=[p/L]-1$. Define
\begin{equation*}
R_{\CC,\Omega,m,a}(i,L)=\#\{ 1\leq x \leq L: N_{\CC,\Omega}(iL+x,I)\equiv a \pmod{m} \}
\end{equation*}
and
\begin{multline*}
R'_{\CC,\Omega,m,b}(i,L) =\#\{ 1\leq x\leq L: \\
\sum_{j=1}^{x}\delta_{\CC,\Omega}(iL+I+j)-\sum_{j=1}^{x}\delta_{\CC,\Omega}(iL+j) \equiv b \pmod{m} \\
~ \forall 1\leq l\leq k \}.
\end{multline*}
Following a similar calculation as in the proof of Theorem \ref{thm1} and Theorem \ref{thm2}, we arrive at
\begin{align}
\sum_{a=0}^{m-1}\left(\Phi_{\CC,\Omega}(m,a)-\frac{1}{m}\right)^2 &\leq \frac{N+1}{p^2}\sum_{a=0}^{m-1}\sum_{i=0}^N \left(R_{\CC,\Omega,m,a}(i,L)-\frac{L}{m}\right)^2 +O\left(\frac{L}{p}\right) \nonumber \\
&= \frac{N+1}{p^2}\sum_{b=0}^{m-1}\sum_{i=0}^N \left(R'_{\CC,\Omega,m,b}(i,L)-\frac{L}{m}\right)^2 +O\left(\frac{L}{p}\right), \label{eqnthm3a}
\end{align}
and
\begin{multline}
\sum_{i=0}^N \left(R'_{\CC,\Omega,m,b}(i,L)-\frac{L}{m}\right)^2 \\
= \frac{1}{m^{2}} \abs{\sum_{\textbf{v},\textbf{v}'\in\{0,1\}^L}\sum_{x=1}^L\sum_{t=1}^{m-1} e_m\left(t\left(\sum_{j=1}^x F(v_{j})-\sum_{j=1}^x F(v_{j}')-b\right)\right)}^2 \\
\times \#M_{\CC,\Omega}(\mathbf{w}), \label{eqnthm3b}
\end{multline}
where $\textbf{w}=(\textbf{v},\textbf{v}')$. This time $\mathbf{x}=(iL+I+1,\ldots,iL+I+L,iL+1,\ldots,iL+L)$, and the condition $p-L>I>L$ guarantees that the entries in $\mathbf{x}$ are disjoint. Applying Proposition \ref{prop21}\eqref{prop21c} and Proposition \ref{prop24} to \eqref{eqnthm3b}, we obtain
\begin{align*}
\sum_{b=0}^{m-1} \sum_{i=0}^N \left(R'_{\CC,\Omega,m,b}(i,L)-\frac{L}{m}\right)^2 &\leq 2^{2L+2}m^4L \left(\frac{\abs{\CI}}{2^{2L} L} +O(2^{2L}\sqrt{p}\log^{2L+1}p)\right) \\
&\leq 4m^4\abs{\CI}+O(m^4 p^{\frac{1}{2}+\varepsilon})
\end{align*}
for any $\varepsilon>0$ (here we used $L=o(\log{p}/\log\log{p})$. Substituting the above back into \eqref{eqnthm3a} and simplifying, we find that
\begin{equation*}
\sum_{a=0}^{m-1}\left(\Phi_{\CC,\Omega}(m,a)-\frac{1}{m}\right)^2 \leq \frac{4m^4N\abs{\CI}}{p^2}+O(m^4/p^{\frac{1}{2}-\varepsilon}) \leq \frac{4m^4}{L(p)} +O(m^4/p^{\frac{1}{2}-\varepsilon}).
\end{equation*}

\end{proof}

\section{An application on the distribution of $\ell$-th power residues and nonresidues}

As an application of our results, we show how they can lead to uniform distribution results of $\ell$-th power residues and nonresidues. First we consider $\ell=2$. Let $\CC$ to be the curve defined by $y^2=x$, and let $L(p)$ be a function that tends to infinity with $p$, but of order $o(\log{p}/\log\log{p})$, and let $I$ be an integer such that $p-L(p)>I>L(p)$. The conditions in Theorem \ref{thm3} are satisfied if we take $\CJ=(\alpha p,\beta p]\subseteq[0,(p-1)/2]$ and $\CI\gg p^{1/2+\delta}$. In our application, we take $\CJ=(0,\beta p]$ ($\beta \leq 1/2$), $\CI=[0,p-1-I]$ (so that we avoid going back to $x=0$).

We say $x\in\mathbb{F}_p^*$ is a $\beta$\textit{-quadratic residue} if $x\equiv y^2 \pmod{p}$ for $y\in(0,\beta p]$, and $x$ is a $\beta$-quadratic nonresidue if it is not a $\beta$-quadratic residue. Recall that $\Omega=\CI\times\CJ$. In this setting, a point $(x,y)$ on $\CC\cap\Omega$ corresponds to the $\beta$-quadratic residue $x$ modulo $p$ (note that we manually excluded $x=0$ in our interval $\CI$). Therefore, the number of points on $\CC\cap\Omega$ with $x\in\CI$ equals the number of $\beta$-quadratic residues in $\CI$. Applying Corollary \ref{cor3} we see that for any positive integer $m$, the number of $\beta$-quadratic residues in $[x_0,x_0+I)$ for $x_0\in\CI$ is uniformly distributed modulo $m$. Since inside an interval of length $I$, the number of $\beta$-residues and nonresidues always sum to $I$, we obtain uniform distribution for the $\beta$-nonresidues as well. More precisely, let $R_{\beta}(x_0,I)$ and $N_{\beta}(x_0,I)$ be the number of $\beta$-residues and nonresidues in the interval $[x_0,x_0+I)$ respectively, and let
\begin{align*}
\Phi_{R,\beta,I}(m,a)&=\frac{1}{p}\#\{ x_0\in[0,p-1-I]: R_{\beta}(x_0,I)\equiv a \pmod{m} \}, \\
\Phi_{N,\beta,I}(m,a)&=\frac{1}{p}\#\{ x_0\in[0,p-1-I]: N_{\beta}(x_0,I)\equiv a \pmod{m} \}.
\end{align*}
Then we have the following.

\begin{cor}\label{cor71}
If $m=o(L(p))^{1/6})$, then
\begin{equation*}
\Phi_{R,\beta,I}(m,a)=\frac{1}{m}+O\left( \frac{m^2}{\sqrt{L(p)}} \right),
\end{equation*}
uniformly for all $0\leq a \leq m-1$. The same holds with $\Phi_{R,\beta,I}(m,a)$ being replaced by $\Phi_{N,\beta,I}(m,a)$.
\end{cor}
Note that if we take $\beta=1/2$, we see that the quadratic residues and nonresidues are uniformly distributed among congruence classes modulo $m$.

If $L(p)$ is a function that tends to infinity with $p$, we fix an interval of length $I=L(p)$, and take $m=[L(p)^{1/7}]$. For any $x_0\in[0,p-1-I]$, there are no $\beta$-quadratic residues (resp. nonresidues) inside the interval $[x_0,x_0+I)$ only if $N_{\beta}(x_0,I)\equiv 0 \pmod{m}$ (resp. $N_{\beta}(x_0,I)\equiv I \pmod{m}$). By Corollary \ref{cor71}, there are at most
\begin{equation*}
p\Phi_{R,\beta,I}(m,a)=\frac{p}{L(p)^{1/7}}+O\left( \frac{pL(p)^{2/7}}{\sqrt{L(p)}} \right)=\frac{p}{L(p)^{1/7}}+O\left( \frac{p}{L(p)^{3/14}} \right)
\end{equation*}
such values of $x_0$. We thus obtain the following result.

\begin{cor}
Let $L(p)$ be an integer function of $p$ that tends to infinity with $p$. For all $x_0\in[0,p-1]$ except possibly $O(p/L(p)^{1/7})$ of them, there is a $\beta$-quadratic residue and a $\beta$-nonresidue inside the interval $[x_0,x_0+L(p))$.
\end{cor}
Taking $\beta=1/2$ gives Corollary \ref{cor4} for the case $\ell=2$.

For $\ell>2$, there is no convenience choice of $\CI$, $\CJ$ such that condition \eqref{cond1} is satisfied, so we use Corollary \ref{cor1} instead. Consider the curve $y^{\ell}=x$, and argue as the case $\ell=2$, we see that $N_{\CC}(x_0,I)$ equals $\ell$ times the number of $\ell$-th power residue in the interval $[x_0,x_0+I)$. Let $\mu$ be an $\ell$-th root of unity and let $R_{\ell,\mu}(x_0,I)$ be the number of $x\in[x_0,x_0+I)$ with $\frac{x}{p}_{\ell}=\mu$. Define
\begin{equation*}
\Phi_{\ell,\mu,I}(m,a)=\frac{1}{p}\#\{ x_0\in[0,p-1]: R_{\ell,\mu}(x_0,I)\equiv a \pmod{m} \}.
\end{equation*}
Invoking Corollary \ref{cor1}, for $GCD(\ell,m)=1$ and $m=o(L(p)^{1/5})$, we have
\begin{equation}\label{eqncor41}
\Phi_{\ell,1,I}(m,a)=\frac{1}{m}+O\left( \sqrt{\frac{m^3\ell^2}{L(p)}} \right).
\end{equation}
For other $\mu\neq 1$, we let $\overline{\mu}$ be its inverse modulo $p$ and consider the curve $y^2=\overline{\mu}x$ to get a similar equation as \eqref{eqncor41} that is true with $\mu$ in place of $1$ in the subindex. We sum them up in the following proposition.
\begin{prop}
If $GCD(m,\ell)=1$ and $m=o(L(p))^{1/5})$, then
\begin{equation*}
\Phi_{\ell,\mu,I}(m,a)=\frac{1}{m}+O\left( \sqrt{\frac{m^3\ell^2}{L(p)}} \right),
\end{equation*}
uniformly for all $0\leq a \leq m-1$ and all $\ell$-th root of unity $\mu$.
\end{prop}

If $L(p)$ is a function that tends to infinity with $p$, we again fix an interval of length $I=L(p)$, and take $m=[L(p)^{1/7}]$ (if this $m$ is not relatively prime to $\ell$, add a small constant to it so that the new $m$ is relatively prime to $\ell$). A similar argument as in the case $\ell=2$ then gives Corollary \ref{cor4} for the case $\ell>2$.

\end{document}